\documentclass[11pt, reqno]{amsart}
\usepackage{amsmath, amssymb, amsfonts, amstext, verbatim, amsthm}

\input xy
\xyoption{all}
\usepackage{setspace}

\usepackage{prettyref}
\newrefformat{prop}{Proposition~\ref{#1}}
\newrefformat{cor}{Corollary~\ref{#1}}
\newrefformat{thm}{Theorem~\ref{#1}}
\newrefformat{lem}{Lemma~\ref{#1}}
\newrefformat{def}{Definition~\ref{#1}}

\usepackage{setspace}
\usepackage{prettyref}
\newrefformat{prop}{Proposition~\ref{#1}}
\newrefformat{cor}{Corollary~\ref{#1}}
\newrefformat{thm}{Theorem~\ref{#1}}
\newrefformat{lem}{Lemma~\ref{#1}}
\newrefformat{def}{Definition~\ref{#1}}

\theoremstyle{plain}
\newtheorem{thm}{Theorem}[section]
\newtheorem{lem}[thm]{Lemma}

\newtheorem{cor}[thm]{Corollary}

\newtheorem{prop}[thm]{Proposition}

\theoremstyle{definition}
\newtheorem{rem}[thm]{Remark}
\newtheorem{defn}[thm]{Definition}

\newcommand{\id}{\operatorname{id}}

\newcommand{\Spec}{\operatorname{Spec}}

\newcommand{\Hom}{\operatorname{Hom}}

\newcommand{\Inj}{\operatorname{Inj}}

\newcommand{\OO}{\mathcal{O}}

\newcommand{\ZZ}{\mathbb{Z}}

\newcommand{\RHom}{\mathbb{R}\!\Hom}
\newcommand{\sheafHom}{\mathcal{H}\!\operatorname{om}}

\newcommand{\Acy}{\operatorname{Acycl^{abs}}}

\newcommand{\MF}{\operatorname{MF}}

\newcommand{\Coh}{\operatorname{Coh}}
\newcommand{\Dabs}{\mathbf{D}^{\operatorname{abs}}\!}

\newcommand{\QC}{\operatorname{QCoh}}
\newcommand{\QCoh}{\operatorname{QCoh}}

\newcommand{\Cone}{\operatorname{Cone}}

\newcommand{\Int}{\operatorname{Int}}
\newcommand{\D}{\mathbf{D}\!}

\newcommand{\Mod}{\operatorname{Mod}}

\newcommand{\sheafRHom}{\mathbb{R}\mathcal{H}\!\operatorname{om}}
\newcommand{\e}{{\operatorname{e}}}

\newcommand{\sheafExtII}{\mathcal{E}\!\operatorname{xt}^{\operatorname{II}}}
\newcommand{\Tot}{\operatorname{Tot}}
\newcommand{\sheafHoch}{\mathcal{H}\!\operatorname{och}}
\newcommand{\disc}{\operatorname{disc}}
\newcommand{\RGamma}{\mathbb{R}\Gamma} 
\newcommand{\cont}{\operatorname{cont}}
\newcommand{\fX}{\mathfrak{X}}

\newcommand{\hatB}{\widehat{\mathcal{B}}}
\newcommand{\C}{\mathcal{C}}
\newcommand{\hatC}{\widehat{\mathcal{C}}}
\renewcommand{\c}{{\operatorname{c}}}
\newcommand{\MFdg}{\operatorname{MF}_{\operatorname{dg}}}
\newcommand{\mf}{\operatorname{mf}}
\newcommand{\MFCech}{\operatorname{MF}_{\operatorname{Cech}}}

\newcommand{\Perf}{\mathfrak{Perf}}
\newcommand{\Sing}{\operatorname{Sing}}
\newcommand{\DSing}{\mathbf{D}^{\operatorname{b}}_{\Sing}}
\newcommand{\Db}{\mathbf{D}^{\operatorname{b}}\!} 
\newcommand{\coker}{\operatorname{coker}}
\newcommand{\thk}{\operatorname{thk}}
\renewcommand{\tilde}{\widetilde}
\newcommand{\op}{{\operatorname{op}}}

\newcommand{\sgn}{\operatorname{sgn}}

\begin{document}

\title{Global Matrix Factorizations}
\author{Kevin H. Lin}
  \address{Department of Mathematics \\ UC Berkeley \\ Berkeley, CA}
  \email{kevin@math.berkeley.edu}
\author{Daniel Pomerleano}
  \address{Department of Mathematics \\ UC Berkeley \\ Berkeley, CA}
  \email{dpomerle@math.berkeley.edu}

\begin{abstract}
We study matrix factorization and curved module categories for Landau--Ginzburg models $(X,W)$ with $X$ a smooth variety, extending parts of the work of Dyckerhoff. Following Positselski, we equip these categories with model category structures. Using results of Rouquier and Orlov, we identify compact generators. Via To\"en's derived Morita theory, we identify Hochschild cohomology with derived endomorphisms of the diagonal curved module; we compute the latter and get the expected result. Finally, we show that our categories are smooth, proper when the singular locus of $W$ is proper, and Calabi--Yau when the total space $X$ is Calabi--Yau.
\end{abstract}

\maketitle

\tableofcontents

\section{Introduction}
Recall the prototypical statement of Homological Mirror Symmetry \cite{Kontsevich}: For every Calabi--Yau manifold $Y$, there is a mirror Calabi--Yau manifold $X$ such that the Fukaya category (resp. derived category) of $Y$ is equivalent to the derived category (resp. Fukaya category) of $X$. Suppose now that $Y$ is not a Calabi--Yau manifold but, say, a smooth toric Fano variety considered as a symplectic manifold. Then the mirror of $X$ is a \emph{complex Landau--Ginzburg model}, or \emph{LG model} for short \cite{Auroux}: a pair $(X,W)$, where $X$ is a complex manifold or variety and $W$ is a holomorphic or regular function.
The function $W$ is called the \emph{superpotential}. The statement of Homological Mirror Symmetry in this situation becomes: The Fukaya category of $Y$ is equivalent to the \emph{matrix factorization category of $(X,W)$}. Moreover when $Y$ has a complex structure and $(X,W)$ has a symplectic structure, the derived category of $Y$ should be equivalent to the Fukaya--Seidel category \cite{Seidel} of $(X,W)$.

If $X=\Spec A$ for a commutative finite type $\mathbb{C}$-algebra $A$ and if $W \in A$ has a single critical value, which we assume without loss of generality to be $0 \in \mathbb{C}$, then the differential $\ZZ/2\ZZ$-graded category of matrix factorizations $\MF(X,W)$ is defined as follows.
This category has as objects
\[P = ( \xymatrix{P_1 \ar@/^/[r]^{p_1} & P_0 \ar@/^/[l]^{p_0} \cr } ) \]
where the $P_i$ are finitely generated projective $A$-modules, and the $p_i$ are $A$-module morphisms satisfying $p_{i+1} \circ p_i = W \cdot \id_{P_i}$.\footnote{The nomenclature ``matrix factorization'' is due to Eisenbud \cite{Eisenbud} and comes from the fact that when the $P_i$ are free modules, the $d_i$ can be thought of as matrices with entries in $A$ that factorize the scalar matrices $W \cdot \id_{P_i}$.} For the morphisms between two matrix factorizations $P$ and $P'$, one takes the $\ZZ/2\ZZ$-graded complex of all $A$-module morphisms 
\[\Hom(P,P') = \bigoplus_{i,j} \Hom_A(P_i,P'_j)\] with grading given by $i+j$ (modulo $2$), and with the differential 
\[\partial: f \longmapsto p' \circ f -(-1)^{|f|} f \circ p \]
for homogeneous $f$. For more details, see \S 3 of \cite{Orlov}. 
Matrix factorizations are also known as \emph{curved $\ZZ/2\ZZ$-graded complexes of finitely generated projective $A$-modules with curvature $W$} \cite{Positselski}. We may truncate this in various ways: \emph{curved complexes of projective modules}, \emph{curved projective modules}, etc.
More generally, we will consider \emph{curved modules} of various kinds, not necessarily finitely generated or projective. 

When $A$ is moreover regular and local, and if $W$ has an isolated singularity at the unique closed point of $\Spec A$, Dyckerhoff \cite{Dyckerhoff} has proven that $\MF(X,W)$ is a smooth and proper Calabi--Yau category satisfying the Hodge-to-de Rham (i.e.\ Hochschild-to-periodic cyclic) degeneration property, and thus a choice of splitting for the degeneration of the spectral sequence gives rise to a 2D TQFT that extends to the Deligne--Mumford boundary \cite{KS, KKP}.

%By \cite{Orlov3}, these results also hold for LG models whose superpotentials have isolated singularities.

%\begin{ex}
%The standard LG model which is mirror to the projective space $\mathbb{P}^n$ has affine total space $(\mathbb{C}^\ast)^n$ and superpotential $W = x_1 + \cdots + x_n + \frac{q}{x_1\cdots x_n}$, where $q$ is a constant. This superpotential has isolated singularities.
%\end{ex}
%More generally, though, one is interested in cases where the singularities of $W$ are not necessarily isolated, and where the total space $X$ is not necessarily affine.
%\begin{ex}
%Let $E$ be a vector bundle over a variety $X$, and $s$ a section of $E$. We can view $s$ as a function $W$ on the total space $\Tot(E^\vee)$. Then the matrix factorization category of the LG model $(\Tot(E^\vee), W)$ should be related to the derived category of the zero locus $Z(s) \subset X$. Similar examples of LG models are constructed and studied in \cite{GKR}.
%\end{ex}

In this paper, we extend some parts of the theory of matrix factorizations to the case of Landau--Ginzburg models $(X,W)$ where $X$ is not necessarily affine. So let $X$ be a smooth variety over $\mathbb{C}$, and let $W$ be a regular function which defines a flat map $X \to \mathbb{A}^1_\mathbb{C}$. Replacing $A$-modules with $\OO_X$-modules, the above definition of matrix factorizations still makes sense --- to be precise, \emph{matrix factorizations} are now defined to be curved complexes of \emph{locally free sheaves of finite type}. However, as is briefly discussed in \S3.2 of \cite{KKP}, the ``correct'' definition of the matrix factorization category in the non-affine situation should take into account the non-vanishing of higher sheaf cohomology. Roughly speaking, this means that we should replace the complex $\Hom(P,P')$ with some form of a derived complex $\RHom(P,P')$, for instance via a \v{C}ech or Dolbeault resolution of the sheaf complex $\sheafHom_{\mathcal{O}_X}(P,P')$. 

To make the above precise, we consider in section 2 of this paper the category $\QC(X,W)$ of curved complexes of \emph{quasi-coherent} sheaves. We equip this category with a model category structure for which fibrant objects are curved complexes of injective sheaves. This gives rise to the dg category $\Inj(X,W)$, which is a dg enhancement of the absolute derived category $\Dabs \QC(X,W)$.
Via fibrant replacement, we define the derived complex $\RHom(P,P')$ of morphisms and thusly the ``correct'' matrix factorization dg category $\MFdg(X,W)$. Furthermore, we show that matrix factorizations are compact when considered as objects of $\Dabs \QC(X,W)$ and that the idempotent completion of the subcategory thereof recovers $\Dabs\QC(X,W)_{\c}$, the full subcategory of all compact objects. This section relies on Positselski's homological theory of curved modules \cite{Positselski} and forms the technical foundation for our paper.

In section 3, we will compute the Hochshild cohomology of $\Inj(X,W)$ and hence that of $\MFdg(X,W)$, yielding a result which was anticipated in \cite{KKP}. To this end, like Dyckerhoff, we identify a compact generator of the category --- all of our subsequent results rely on this important observation. Dyckerhoff identifies generators for matrix factorization categories in the local situation, but we identify generators in the global situation for the corresponding \emph{derived categories of singularities} of the zero fibers using the work of Rouquier \cite{Rou}. A theorem of Orlov \cite{Orlov2} says that the two categories are the same. On the other hand, Dyckerhoff is able to explicitly compute the endomorphism dg algebras of his generators. In our case, there is no clear way in general to associate an explicit matrix factorization to a generator of the derived category of singularities, and so no clear way to compute the endomorphism dg algebra. However, the abstract identification of the category with the dg derived category of the endomorphism dg algebra of the generator is enough to apply the derived Morita theory of \cite{Toen}. To reach our Hochschild cohomology result, we next take a detour into the work of \cite{Yekutieli} on the global Hochschild--Kostant--Rosenberg theorem, and we employ the calculations of \cite{CaldararuTu}. It also follows from results in this section that $\MFdg(X,W)$ is smooth, and that it is proper when the singular locus of $W$ is proper. 

Finally, in section 4, using Grothendieck duality \cite{RD}, we show that if the Landau--Ginzburg model $(X,W)$ satisfies the condition that $X$ is Calabi--Yau, then $\MFdg(X,W)$ is a Calabi--Yau category. We remark that our proof of the Calabi--Yau condition on the category mimics Dyckerhoff's proof in the affine case; however, we are able to identify explicitly how the Calabi--Yau condition on the space $X$ comes into play. This is not immediately transparent in Dyckerhoff's proof, since in his local situation the Calabi--Yau condition on $X$ is automatic.

For ease of notation and exposition, we always assume that our superpotentials $W$ have a single critical value $0 \in \mathbb{C}$. If there are multiple critical values $c_i$, then the results will all still hold by considering the product $\prod_i \MF(X, W-c_i)$ instead of $\MF(X,W)$. Furthermore, unless specified otherwise, when we say \emph{dg category} we will always mean differential $\ZZ/2\ZZ$-graded category, that is, a category enriched over the category of $\ZZ/2\ZZ$-graded complexes of $\mathbb{C}$-vector spaces. More generally, all of our graded objects are $\ZZ/2\ZZ$-graded objects. We work over the field $\mathbb{C}$, since this is the situation of primary interest in applications, but we remark that all results still hold over any field of characteristic zero.

We note that Anatoly Preygel has independently proved results similar to our results in this paper, using an exciting new and different approach involving derived algebraic geometry \cite{Preygel}.

\section{Curved quasi-coherent sheaves and matrix factorizations}

Let $X$ be a smooth variety over $\mathbb{C}$ and $W$ a regular function such that the corresponding morphism $X \to \mathbb{A}^1_\mathbb{C}$ is flat. We now consider the dg category of curved complexes of quasi-coherent sheaves $\QCoh(X,W)$, that is, the category with objects \[E = ( \xymatrix{{E}_1 \ar@/^/[r]^{e_1} &{E}_0\ar@/^/[l]^{e_0} \cr} )\]  where the ${E}_i$ are quasi-coherent sheaves of $\OO_X$-modules and the $e_i$ are morphisms of $\OO_X$-modules satisfying, as before, $e_{i+1} \circ e_i = W \cdot \id_{E_i}$. The morphism complexes are defined exactly as before except with $\Hom_{\OO_X}$ rather than $\Hom_A$. We will denote by $E[1]$ the curved complex 
\[ ( \xymatrix{{E}_0 \ar@/^/[r]^{ - e_0} &{E}_1\ar@/^/[l]^{ - e_1} \cr} ) . \]
Furthermore, one can define the cone of a morphism and a class of exact triangles in $\QCoh(X,W)$ which together with the shift functor $E \mapsto E[1]$ makes the homotopy category $[\QCoh(X,W)] = H^0(\QCoh(X,W))$ a triangulated category. For more details, please refer to \cite{Orlov, Orlov2}. 

More generally, given any dg category $\mathcal{C}$ of curved objects, we will let $[\mathcal{C}]$ denote the homotopy category of $\mathcal{C}$ with triangulated category structure defined in the same way. Also we will use the functor $E \mapsto E^\#$ sending a curved object to the underlying graded object gotten by forgetting the maps $e_0$ and $e_1$.

\begin{defn} 
Denote by $\Acy[\QC(X,W)] \subset[\QC(X,W)]$ the thick triangulated subcategory generated by\footnote{This means that we take recursively all shifts, cones, and direct summands \cite{Rou}.} 
the total curved complexes of exact triples of curved quasi-coherent $\mathcal{O}_X$-modules. Objects of $\Acy[\QC(X,W)]$ are called \emph{acyclic}. The triangulated category $\Dabs\QC (X,W)$ is defined to be the quotient triangulated category \[\frac{[\QC(X,W)]}{\Acy[\QC(X,W)]}.\] We call this category the \emph{absolute derived category}. This definition is also used in \cite{Positselski, Orlov2}.
\end{defn}

\begin{rem}
Note that in our curved situation, we are unable to define the derived category in the usual way by inverting quasi-isomorphisms --- we cannot speak of cohomology of a \emph{curved} complex, since we don't have ``$d^2=0$'', and thus we cannot speak of quasi-isomorphism of curved complexes. Similarly, the usual notion of acyclicity does not make sense. However, note that in the case of ordinary uncurved complexes of sheaves, the total complex of an exact sequence of complexes is acyclic. This motivates the definitions of acyclicity and absolute derived category. 
\end{rem}

\begin{lem}
\label{lem:triangulated}
Let H be a triangulated category and $A, F$ be full triangulated subcategories. Then the natural functor $F/(A \cap F) \to H/A$ is an equivalence of triangulated
categories if for any object $X \subset H$ there exists an object $Y \in F$ together with a morphism
$X \to Y $ in H such that a cone of that morphism belongs to $A$.\end{lem}

\begin{prop}
\label{prop:injective}Denote by $\Inj(X,W)$ the full subcategory of $\QC(X,W)$ consisting of curved complexes of injective quasi-coherent sheaves. The natural functor $[\Inj(X,W)] \to \Dabs\QC(X,W)$ is an equivalence of triangulated categories. Therefore the category $\Inj(X,W)$ defines a dg enhancement of $\Dabs\QC(X,W).$\end{prop}

\begin{proof} This is a scheme theoretic version of Theorems 3.5 and 3.6 of \cite{Positselski} and is proved in exactly the same way. We give a sketch of the proof. We wish to apply the previous lemma and so we proceed in two steps. 

The first step is very general --- we claim that if $B \in \Acy[\QC(X,W)]$ and $I$ is a curved complex of injective sheaves, then $\Hom(B,I)$ is an acyclic complex. Indeed if $B$ is the total curved module of an exact sequence of curved modules \[0 \to L \to M \to N \to 0,\] then $\Hom(B,I)$ is the total complex of the exact sequence of complexes \[0 \to \Hom(N,I) \to \Hom(M,I) \to \Hom(L,I) \to 0,\] so it is acyclic. 
Since $\Acy[\QC(X,W)]$ is the thick triangulated subcategory generated by such $B$, the claim follows. We see immediately that $\Acy[\QC(X,W)]\cap \Inj(X,W) = 0$.

It remains to show that for each $B \in \QC(X,W)$, there is a morphism $r : B \to J$ such that $J \in \Inj(X,W)$ and $\Cone(r) \in \Acy[\QC(X,W)]$. Indeed, there is an embedding of any curved complex of quasi-coherent sheaves $B$ into a curved complex of injectives $G_0$. To see this, note that the underlying graded sheaf $B^\#$ embeds into an injective graded quasi-coherent sheaf $I_0$, as the category of quasi-coherent sheaves has enough injectives. One then takes the curved complex $G^-(I_0)$ of quasi-coherent sheaves cofreely cogenerated by $I_0$ (see the proof of Theorem 3.6 of \cite{Positselski}), and one checks that $B$ embeds into $G_0 := G^-(I_0)$ and that $G^-(I_0)^\#$ is injective. Let $H_0$ be the cokernel $G_0/B$, and similarly we construct a curved complex $G_1 = G^-(I_1)$ of injectives into which $H_0$ embeds. Proceeding inductively, we obtain a resolution $0 \to B \to G_0 \to G_1 \to \cdots$ of $B$ by curved complexes of injectives.
However, since $X$ is smooth, the category $\QCoh(X)$ of quasi-coherent sheaves has finite homological dimension, and hence for some finite $n$ we must have that the underlying graded sheaf $H_n^\#$ of the cokernel $H_n = G_{n} / G_{n-1}$ is injective. Let $J$ be the total curved module of the exact complex of curved modules \[G_0 \to G_1 \to \cdots \to G_n \to H_n.\] We are finished.
\end{proof}
 
\begin{rem}
Dyckerhoff \cite{Dyckerhoff} considers a regular local $k$-algebra $R$ with maximal ideal $\mathfrak{m}$ and residue field $R/\mathfrak{m} = k$. 
He takes a superpotential $W \in R$ with isolated singularity at the closed point $\mathfrak{m}$, and he considers the category $\MF^{\infty}(R,W)$ consisting of curved complexes of \emph{projective} $R$-modules of arbitrary rank. In \cite{Positselski}, it is proved that $[\MF^{\infty}(R,W)]$ and $\Dabs\QC (R,W)$ are equivalent as triangulated categories. In our case, we are forced to use curved complexes of \emph{injective} modules because there are not enough projectives in the global situation.
\end{rem}

%It is possible to prove the following strengthening of the above proposition.

\begin{thm} There is a model category structure \cite{Hovey} on $\QC(X,W)$ where a morphism is a weak equivalence if its cone is acyclic; a morphism is cofibrant if it is monic and it is fibrant if it is epic and its kernel is a curved complex of injective sheaves; fibrant objects are curved complexes of injective sheaves.\end{thm} 

In the language of \cite{Toen}, we have $\Inj(X,W)=\Int(\QC(X,W))$, where $\Int(-)$ denotes the full subcategory consisting of objects which are both fibrant and cofibrant. The proof of the above theorem is again similar to Positselski's discussion in the affine case and we omit it since we will not need the full strength of the theorem in the rest of this paper. 

In view of the theorem, it is interesting to speculate that the functor $\Inj(-,W)$ defines a sheaf of dg categories on the Zariski site of $X$ which is fibrant for an appropriate model category structure, but the current literature on sheaves of dg categories appears to be insufficient for making a precise conjecture. 

Now, given two curved complexes of quasi-coherent sheaves ${F}$ and ${F'}$, we have an \emph{uncurved} complex of sheaves $\sheafHom ({F}, {F'})$ which is defined by
\[ U \mapsto \Hom_{\OO_U}(F|_U,F'|_U). \]

\begin{defn}
\label{def:RHom}
We have a derived functor
\[ \sheafRHom : \Dabs\QC(X,W) \times \Dabs\QC(X,W) \to \D \Mod(\OO_X). \]
It is defined by first doing at least one of the following: 
\begin{enumerate}
\item replacing the second argument by a weakly equivalent curved complex of injectives
\item if possible, replacing the first argument by a weakly equivalent curved complex of locally free sheaves of finite rank

\end{enumerate}
and then taking $\sheafHom$.
We have another derived functor
\[ \RHom : \Dabs\QC(X,W) \times \Dabs\QC(X,W) \to \D\Mod(\mathbb{C})\]
defined by first doing (1) and then taking $\Hom$.
\end{defn} 

We now define two different categories of matrix factorizations (same objects, different morphisms).

\begin{defn} Define $\mf(X,W)$, $\Acy[ \mf(X,W)]$, and $\Dabs \mf(X,W)$ in the same way as we defined the analogous respective $\QCoh$ entities above, except here the objects are curved complexes of locally free sheaves of finite type, i.e. curved vector bundles, i.e. matrix factorizations.
\end{defn}

\begin{defn}
Denote by $\MFdg(X,W)$ the full dg subcategory of $\Inj(X,W)$ consisting of objects weakly equivalent to matrix factorizations.
\end{defn}

\begin{rem}
What we call $\Dabs \mf (X,W)$ agrees with what Orlov calls $\MF_0(X,W)$ in \cite{Orlov2}. Recall from the introduction that we are assuming in this paper that $W$ has only one singular value $0 \in \mathbb{C}$.
\end{rem}

In some arguments it will be convenient to use a third definition --- a \v{C}ech model of $\MFdg(X,W)$. Let $\mathfrak{U} = \{ U_i = \Spec A_i \}$ be a finite covering of $X$ by affine subsets. We follow the notation of \S III.4 of \cite{Hartshorne}, and we write $\C^\bullet(\mathfrak{U},F)$ for the sheaf \v{C}ech complex of a sheaf $F$. We define the dg category $\MFCech(X,W)$ as follows: The objects are matrix factorizations; the morphisms $\Hom_{\MFCech}(P,P')$ are given by the global sections of the total complex of the double complex $\C^\bullet(\mathfrak{U}, \sheafHom(P,P'))$ with the first differential being the \v{C}ech differential and the second differential induced by that of $\sheafHom(P,P')$. Although $\MFCech(X,W)$ depends on the covering $\mathfrak{U}$, we suppress this from the notation because different coverings yield weakly equivalent dg categories\footnote{The category of dg categories has a model category structure for which weak equivalences are quasi-equivalences of dg categories.}. It is a tedious but standard consideration to see the following:

\begin{prop} We have a weak equivalence $\MFCech(X,W) \to \MFdg(X,W)$ of dg categories. \end{prop}
 
\begin{defn}[\cite{Orlov}]
For any variety $Y$ over $\mathbb{C}$, we denote by $\Db \Coh(Y)$ the bounded derived category of coherent sheaves on $Y$, and we denote by $\Perf(Y)$ the full triangulated subcategory of perfect complexes. We set 
\[ \DSing(Y) = \frac{\Db \Coh (Y)}{\mathfrak{Perf}(Y)} .\]
\end{defn}

\begin{prop}  $[\MFdg(X,W)]$ is equivalent to $\DSing(X_0)$, where $X_0$ denotes the fiber $W^{-1}(0)$. \end{prop}
\begin{proof} There is a natural triangulated functor $\coker: [\mf(X,W)] \to \DSing (X_0)$ given by
\[ P = ( \xymatrix{P_1 \ar@/^/[r]^{p_1} & P_0 \ar@/^/[l]^{p_0} \cr } ) \longmapsto \coker(p_1) =: \coker(P) . \]

Let $\{ U_i \}$ be as above an affine open cover of $X$.
Consider an object $P$ of $\mf(X,W)$ whose image in the homotopy category lies in $\Acy[ \mf(X,W)]$. Then $P|_{U_i}$ is in $\Acy[ \mf(U_i,W)]$. By an argument similar to the first part of the proof of \prettyref{prop:injective} (except in this situation consider projectives instead of injectives), this subcategory is $0$, which means that $P|_{U_i}$ is contractible and hence its cokernel is locally free \cite{Orlov}. Since this holds for each $U_i$ we conclude that $\coker({P})$ is locally free and therefore vanishes in $\DSing(X_0)$. Thus the $\coker$ functor factors through $\Dabs \mf (X,W)$, and \cite{Orlov2} proves that the induced functor $\Dabs \mf (X,W) \to \DSing(X_0)$ is an equivalence of triangulated categories.

To prove the proposition, it suffices to show that the natural functor \[\Dabs\mf(X,W) \to \Dabs\QC(X,W)\] is fully faithful. For this purpose, it is useful to consider the categories $\Coh(X,W)$, $\Acy[ \Coh(X,W)]$, and $\Dabs\Coh(X,W)$ defined in the same way as the respective $\mf$ and $\QCoh$ entities. By Exercise II.5.15 of \cite{Hartshorne}, any morphism from a curved coherent sheaf $F \in \Coh(X,W)$ to an acyclic curved quasi-coherent sheaf $A \in \Acy [\QC(X,W)]$ factors through an acyclic curved coherent sheaf $A' \in \Acy [\Coh(X,W)]$. It follows that $\Dabs \Coh (X,W) \to \Dabs \QCoh(X,W)$ is fully faithful.

It remains to show that $\Dabs \mf (X,W) \to \Dabs \Coh (X,W)$ is fully faithful. To see this, note that since we are on a smooth scheme, any coherent sheaf has a finite resolution by vector bundles. 
Therefore, following a similar argument as in \prettyref{prop:injective}, for any curved coherent sheaf $C$ we can produce a triangle $F\to C\to A$ where $F$ is a matrix factorization and $A$ is acyclic. It follows from the dual version of \prettyref{lem:triangulated} that \[[\mf(X,W)]/([\mf(X,W)]\cap \Acy[\Coh(X,W)]) \cong \Dabs\Coh(X,W).\] The thick subcategory $[\mf(X,W)]\cap \Acy[\Coh(X,W)]$ can be identified with the thick subcategory $\Acy[\mf(X,W)]$. We see this as follows. Taking $\coker (P)$ of objects $P$ of the former category gives the zero object in $\DSing(X_0)$ by the same local argument explained at the beginning of this proof. By Orlov's result mentioned above, it follows that $P$ must have been equivalent to an object in $\Acy[\mf(X,W)]$ to begin with.
\end{proof}

\begin{rem}
The cokernel of an acyclic curved coherent sheaf does not necessarily represent the zero object in $\DSing(X_0)$. However, one may still define a derived cokernel functor $\Dabs \Coh (X,W) \to \DSing(X_0)$ by composing the functor $\Dabs \mf (X,W) \to \DSing(X_0)$ with an inverse to the equivalence $\Dabs \mf(X,W) \to \Dabs \Coh (X,W)$ from above.
\end{rem}

\begin{prop} Objects of $[\MFdg(X,W)]$ are compact as objects of the absolute derived category $\Dabs \QC(X,W)$. 
\end{prop}
\begin{proof}
Let $P$ be a matrix factorization and $Q$ an arbitrary curved quasi-coherent sheaf. 
It is standard to see that $\RHom(P,Q)$ can be computed using the complex $\Gamma \Tot \C^\bullet ( \mathfrak{U} , \sheafHom(P,Q))$.

Since the $U_i$ (and their intersections $U_{ij}$, etc.) are affine, it follows that the restrictions $P|_{U_i}$ are compact in $\Dabs \QC(U_i,W)$ by \cite{Positselski} (and analogously for the intersections $U_{ij}$, etc.). Let $Q = \bigoplus_i Q_i$ be a direct sum of curved quasi-coherent sheaves. We have \[\Gamma \Tot \C^\bullet ( \mathfrak{U} , \sheafHom(P, \bigoplus_i Q_i)) \cong \Gamma \Tot \C^\bullet( \mathfrak{U},\bigoplus_i \sheafHom(P,Q_i)),\] because the restrictions $P|_{U_i}$, $P|_{U_{ij}}$, etc. are compact, and so finally we have \[\Gamma \Tot \C^\bullet( \mathfrak{U},\bigoplus_i \sheafHom(P,Q_i)) = \bigoplus_i \Gamma \Tot \C^\bullet ( \mathfrak{U} , \sheafHom(P,Q_i)).\]
This completes the proof.
\end{proof}

For what follows, we need the following well-known lemma \cite{BonVDB}:

\begin{lem} Let $\mathcal{T}$ be a triangulated category with arbitrary direct sums and which is compactly generated by a set of objects $C$. Then the set of compact objects of $\mathcal{T}$ is $C^{\thk}$, the thick closure of $C$.\end{lem}

\begin{prop} $\overline{[\MFdg(X,W)]}\cong \Dabs\QC(X,W)_{\c}$, where the notation $\overline{\mathcal{C}}$ means idempotent completion of a category $\mathcal{C}$, and the notation $\mathcal{C}_\c$ means the full subcategory of $\mathcal{C}$ consisting of objects whose image in the triangulated category $[\mathcal{C}]$ is compact.
\end{prop}

\begin{proof} By the above lemma it suffices to prove that $[\MFdg(X,W)] \cong \Dabs \mf(X,W) \cong \Dabs \Coh (X,W)$ generates $\Dabs\QC(X,W)$. What we want to prove is the global version of Theorem 2 on page 43 of \cite{Positselski} and the proof is very similar. Let $J$ be an object of $\Inj(X,W)$. By the standard Bousfield localization argument, what we have to show is that if $\Hom(B,J)$ is acyclic for every coherent curved module $B$, then $J$ is contractible, meaning that it is weakly equivalent to the zero object.

Consider the ordered set of pairs $(C,h)$, where $C$ is a curved quasi-coherent subsheaf of $J$ and $h$ is a contracting homotopy for the inclusion $C \hookrightarrow J$. Using Zorn's lemma, let $(M,h)$ be a maximal such pair. 
We show that if $M \neq J$, then $M \hookrightarrow J$ factors through some $M' \hookrightarrow J$, and the contracting homotopy $h$ extends to a contracting homotopy $h'$ for $M' \hookrightarrow J$. From here the result follows.

So suppose $M \neq J$. Then again using Exercise II.5.15 of \cite{Hartshorne}, we can find a curved quasi-coherent subsheaf $M'$ of $J$ such that $M'$ strictly contains $M$ and the quotient $M'/M$ is coherent. Producing the contracting homotopy proceeds exactly as in \cite{Positselski}.\end{proof}

The following will be used in the next section:

\begin{lem} 
\label{lem:coker}
Let $F$ be a coherent sheaf on $W^{-1}(0) = X_0$ considered as an object of $\Coh(X,W)$. Suppose $P$ is a matrix factorization and $f: P \to F$ is a morphism of curved sheaves such that $\Cone(f)$ is acyclic. Then $\coker( P) \cong F$ in $\DSing(X_0)$. Moreover, such a $P$ exists.
\end{lem}

\begin{proof} 
We know that $P \cong F$ in $\Dabs \Coh (X,W)$. First we check that the result holds if $F$, as a coherent sheaf, is \emph{maximal Cohen--Macaulay}, which means that $\mathcal{E}\!\operatorname{xt}^i(F,\OO_{X_0}) = 0$ for $i > 0$. To see this, note that there is a length two resolution of $F$ by locally free sheaves on $X$ (see the proof of Theorem 3.9 in \cite{Orlov})
\[ 0 \to Q_1 \to Q_0 \to F \to 0. \]

Let $G^+(Q_0)$ be the free curved module generated by $Q_0$ (see again Theorem 3.6 of \cite{Positselski}).
We have a surjection of curved sheaves $G^+(Q_0)\to {F}$ whose kernel is isomorphic to ${Q}[1]$, where \[{Q}=(\xymatrix{Q_1 \ar@/^/[r]^{q_1} & Q_0 \ar@/^/[l]^{q_0} \cr } ) , \] with $q_1$ the inclusion map and $q_0$ the homotopy expressing the fact that $W$ kills $F$. 

We clearly have $\coker(Q) = \coker(F) = F$. Since $G^+(Q_0)$ is contractible, we have an isomorphism $Q \cong F$ in $\Dabs \Coh (X,W)$. Hence we have $P \cong F \cong Q$ in $\Dabs\Coh(X,W)$. Previously we checked that the functor \[\Dabs\mf(X,W) \to \Dabs\Coh(X,W)\] is fully faithful, and hence $P \cong Q$ in $\Dabs \mf(X,W)$. Thus $\coker(P) \cong \coker(Q) = F$ in $\DSing(X_0)$.

For the general case, for any coherent sheaf ${F}$, there is a resolution $$0\to F' \to F_r \to \cdots \to F_2\to{F}_1\to {F} \to 0$$ where the ${F}_i$ are locally free and ${F}'$ is maximal Cohen--Macaulay. We assume without loss of generality that $r$ is even; if it is odd we simply consider an additional syzygy $F_{r+1} \to F_r$ and take $F'$ to be its kernel. Then we conclude that ${F} \cong{F}'$ in $\Dabs\Coh(X,W)$ as well as in $\DSing(X_0)$ and so the lemma is proven.\end{proof}

\begin{rem} 
After receiving an earlier version of this article, Leonid Positselski has expanded on the ideas in this section to produce a new proof of Orlov's equivalence \cite{Positselski2}. 
\end{rem} 

\section{Compact generators and Hochschild (co)homology}

Let $X$ be as above a smooth variety over $\mathbb{C}$, and let $W$ be an arbitrary superpotential. The purpose of this section is to prove the following theorem:
\begin{thm} 
\label{thm:HH}
The Hochschild cohomology of the category $\MFdg(X,W)$ is given by $\RGamma(\Lambda^\bullet T_X,[W,-]),$ where $[- , -]$ denotes the Schouten--Nijenhuis bracket.
\end{thm} 

\begin{rem} One can similarly determine that the Hochschild homology is given by $\RGamma(\Omega^\bullet_X, dW\wedge)$. We focus on Hochschild cohomology, both in the interest of brevity and because in the case when $X$ is Calabi--Yau --- which is actually our primary case of interest --- the Hochschild homology result follows by section 4 of this paper.\end{rem} 

The Hochschild cohomology of a dg category can be defined as the derived endomorphisms of the identity functor of the category \cite{Toen}. The Hochschild cohomology of $\MFdg(X,W)$ is the same as that of $\Inj(X,W)$. Consider the category of endofunctors of $\Inj(X,W)$, and consider the full subcategory consisting of continuous functors. We will identify this full subcategory with the category $\Inj(X \times X, \widetilde{W})$, where \[\tilde{W} := \pi_1^\ast(W) - \pi_2^\ast(W).\] Furthermore, we will identify the identity endofunctor with the diagonal curved complex $\Delta \in \Inj(X \times X, \widetilde{W})$, that is, the structure sheaf $\mathcal{O}_\Delta$ of the diagonal $X \hookrightarrow X \times X$ considered as a curved complex. 
We then compute the Hochschild cohomology of $\MFdg(X,W)$ by computing the derived endomorphisms of $\Delta$.

\begin{lem} $\RHom(\Delta,\Delta) \cong \RGamma(\Lambda^\bullet T_X, [W, -]).$\end{lem}

\begin{proof} 
We have a functor \cite{PP}
\[\sheafExtII : \Dabs \QC(X,W) \times \Dabs \QC(X,W) \to \D \QC(X). \]
which is defined as follows --- first do at least one of the following two things: 
\begin{enumerate}
\item replace the second argument with a complex $I^\bullet$ of curved complexes of injective sheaves,
\item if possible, replace the first argument with a complex $P^\bullet$ of curved complexes of locally free sheaves of finite rank
\end{enumerate}
then take their $\sheafHom$, and then finally take the direct sum total complex ${\Tot}^{\oplus}$ of the resulting double or triple complex. Because $X$ is smooth and so $\QC(X)$ has finite homological dimension, we can choose such resolutions to have finite length, and thus we have that $\sheafExtII (\Delta,\Delta)$ and $\sheafRHom^\bullet(\Delta,\Delta)$ agree.

Our proof is essentially a curved adaptation of \cite{Yekutieli}.
Let $\fX^q$ be the formal completion of $X^q = X \times \cdots \times X$ along the diagonal $X$. For a commutative algebra $A$, denote by $B_q(A)$ the $q$th term $A \otimes A^{\otimes q} \otimes A$ in the standard bar complex $B(A)$. Let $\hatB_q(A)$ be the $I_q$-adic completion of $B_q(A)$, where $I_q$ is the kernel of the map $B_q(A) \to A$ defined by $a_0 \otimes \cdots \otimes a_{q+1} \mapsto a_0 \cdots a_{q+1}$. On $B(A)$ we have the usual bar complex differential $\partial_B$, and we also have the ``curved'' differential $\partial_W$ which is defined by
\[a_0 \otimes a_1 \otimes \cdots \otimes a_n \mapsto \sum_{i=0}^{n-1} (-1)^{i+1} a_0 \otimes a_1 \otimes \cdots \otimes a_i \otimes W \otimes a_{i+1} \otimes \cdots \otimes a_n. \]
Check that $(\partial_B + \partial_W)^2 = (\tilde{W} \cdot - )$.
It is easy to check that $\partial_B$ and $\partial_W$ are continuous with respect to the $I$-adic topologies.

Define $\hatB_q(X) := \OO_{\fX^{q+2}}$. On an open affine $U = \Spec A \subset X$, we have \[\Gamma(U, \hatB_q(X)) = \hatB_q(A).\] The $\partial_B$ and $\partial_W$ sheafify to give maps $\partial_B : \hatB_q(X) \to \hatB_{q-1}(X)$ and $\partial_W : \hatB_q(X) \to \hatB_{q+1}(X)$. 
Now let $M$ be a curved $\OO_{X^2}$-module with curvature $\tilde{W}$. We denote the \emph{Hochschild cohomology complex of $\OO_X$ with coefficients in $M$} by $\sheafHoch^{\oplus} (\OO_X, M)$, and we define it as follows. It is a $\ZZ/2\ZZ$-graded complex with $i$th component given by
\[ \bigoplus_{p+q = i} \sheafHom^{\cont}_{\OO_{\fX^2}} (\hatB_q(X), M_p). \]
This complex has differential $\partial + \partial_B + \partial_W$, where $\partial$ is induced from $M$, and $\partial_B$ and $\partial_W$ are induced by the respective maps defined above (see page 24 of \cite{PP}).
The superscript ``cont'' denotes continuous morphisms, where we have the adic topology on $\hatB(X)$ and the discrete topology on $M$.

The category $\Mod_{\disc} \OO_{\fX^2}$ (see \S2 of \cite{Yekutieli}) has enough injectives. In fact, it is straightforward to see that these injectives can be chosen to be quasi-coherent as sheaves on $\OO_{X^2}$. Consider $\OO_X$ as an object of this category. Using the construction of \prettyref{prop:injective}, we can then construct a resolution $I^\bullet$ of $\OO_X$ by curved injective quasi-coherent sheaves on $\OO_{X^2}$.

%Following the same argument as in \prettyref{prop:injective}, the curved module $\OO_X$ can be resolved by curved complexes of injective objects in $\Mod_{\disc} \OO_{\fX^2}$. Let $0 \to \OO_X \to I^\bullet$ be such a resolution. Then $0 \to \OO_X \to I^\bullet$ is also a resolution of $\OO_X$, as an $\OO_{X^2}$-module, by curved injective $\OO_{X^2}$-modules.

Therefore, we have 
\[\sheafExtII_{\OO_{X^2}} ( \OO_X, \OO_X) = \Tot^{\oplus} \sheafHom_{\OO_{X^2}} (\OO_X, I^\bullet). \]
%On a locally noetherian scheme $Y$, since sheaves which are injective as quasi-coherent sheaves are also injective as $\OO_Y$-modules (see \cite{RD}), this means that in our situation resolving $\OO_X$ by a complex of curved injective but not necessarily quasi-coherent sheaves gives the same $\sheafExtII$ result as if we had resolved it by a complex of curved injective quasi-coherent sheaves.
Since all of the sheaves involved are discrete, we have
\[ \Tot^\oplus \sheafHom_{\OO_{X^2}} (\OO_X, I^\bullet) = \Tot^\oplus \sheafHom_{\OO_{\fX^2}}^{\cont} (\OO_X, I^\bullet). \]
Consider the bicomplex $\sheafHoch^\oplus (\OO_X , I^\bullet)$ and the total complex obtained by taking direct sums of the diagonals of this bicomplex.  Then we consider two maps to this total complex:
\[ \Tot^\oplus \sheafHom_{\OO_{\fX^2}}^{\cont} (\OO_X, I^\bullet) \to \Tot^\oplus \sheafHoch^\oplus (\OO_X, I^\bullet)\]
and
\[
\left(\bigoplus_q \sheafHom_{\OO_{\fX^2}}^{\cont} (\hatB_q(X), \OO_X) , \partial_B + \partial_W \right) \to \Tot^\oplus \sheafHoch^\oplus (\OO_X, I^\bullet). \]
The second map is induced by the morphism $\hatB(X) \to \OO_X$ and is a quasi-isomorphism by a spectral sequence argument. The first map is induced by $\OO_X \to I^\bullet$ and is a quasi-isomorphism by Lemma 2.7 of \cite{Yekutieli}, which states that that when $X$ is smooth over $\mathbb{C}$, the functor 
\[\sheafHom_{\OO_{\fX^2}}^{\cont} (\hatB_q(X), - ) : \Mod_{\disc} \OO_{\fX^2} \to \Mod_{\disc} \OO_{\fX^2} \]
is exact. Our argument here parallels the argument on page 25 of \cite{PP}.

Define $\hatC(X) := \hatB(X) \otimes_{\OO_{\fX^2}} \OO_X$, with induced differential also denoted by $\partial_B + \partial_W$. Then we have an identification of complexes
\[ \left(\bigoplus_q \sheafHom_{\OO_{\fX^2}}^{\cont} (\hatB_q(X), \OO_X) , \partial_B + \partial_W \right) \cong \left( \bigoplus_q \sheafHom_{\OO_{X}}^{\cont} (\hatC_q(X), \OO_X) , \partial_B + \partial_W \right). \]
We now note that there is a quasi-isomorphism 
\[ \pi:(\Lambda^\bullet T_X, 0) \to \left(\bigoplus_q \sheafHom_{\OO_X}^{\cont}(\hatC_{q}(X),\OO_X), \partial_B\right).\]
On an affine subscheme $\Spec A$, each graded component of the right hand side can be identified with polydifferential operators, namely the subcomplex of $\Hom_k(A^{\otimes q} , A)$ consisting of maps that are differential operators in each factor, and the isomorphism has the form 
\[\pi(v_i\wedge  \cdots \wedge v_q)(a_1 \otimes \cdots \otimes a_q) = \frac{1}{q!} \sum_{\sigma \in S_q}  \sgn(\sigma) v_{\sigma(1)} (a_1)  \cdots v_{\sigma(q)}(a_q).\]
One computes explicity in these local coordinates that $\pi([W,-]) = \partial_ W (\pi(-))$ and thus we get an induced map of complexes 
\[ \pi:(\Lambda^\bullet T_X,[W,-])\to \left(\bigoplus_q \sheafHom_{\OO_X}^{\cont}(\hatC_q, \OO_X), \partial_B + \partial_W\right).\]
We conclude using exactly the same spectral sequence argument as in \cite{CaldararuTu} in the affine case that $\pi$ is a quasi-isomorphism. 
\end{proof}

To complete \prettyref{thm:HH}, we need the following result, which uses the language of \cite{Toen}:

\begin{thm} 
\label{thm:functor}
We have \[\RHom_{\c}(\Inj(X_1,W_1),\Inj(X_2,W_2))\cong \Inj(X_1 \times X_2,\pi_1^*(W_1)-\pi_2^*(W_2)).\] Here $\RHom_{\c}$ denotes continuous functors, i.e. functors which commute with arbitrary direct sums. When $X_1=X_2$ and $W_1=W_2$, then the induced equivalence of homotopy categories identifies the identity endofunctor with the diagonal curved sheaf $\Delta$ as an object of $\Dabs \QC (X \times X, \tilde{W})$.
\end{thm}

We will use the following theorem, which follows by results in \S7 of \cite{Rou}.

\begin{thm} If $Z$ is a generator for $\Db\Coh(\Sing(X))$
and $Y$ is a generator of $\Perf(X)$, then $i_* Z \oplus Y$ generates $\Db\Coh(X)$. Here $\Sing(X)$ denotes the singular locus of $X$, and $i$ denotes the inclusion $\Sing(X) \hookrightarrow X$.
\end{thm}

It follows that generators of $\Db\Coh(\Sing(X))$ are also generators of $\DSing(X)$. We hence obtain a new proof of a result of Dyckerhoff.

\begin{cor}[\cite{Dyckerhoff}]
If $W$ has exactly one isolated singularity, then the residue field $\mathbb{C}$ of the singularity is a generator of the category $\DSing(W^{-1}(0)) \cong [\MFdg(X,W)]$.
\end{cor}

\begin{proof}
The structure sheaf is a generator of $\Db\Coh(\Spec \mathbb{C})$.
\end{proof}

We will also use the following theorem, which can be proven explicitly for the generators constructed inductively in \cite{Rou}, but in the hope that it might be useful in future work, we give a more general statement. The proof was outlined to us by Raphael Rouquier.

\begin{thm} 
\label{thm:generator}
If $E$ is a generator of $\Db\Coh(X)$ and $F$ is a generator of $\Db\Coh(Y)$ then $E\otimes F$ is a generator of $\Db\Coh(X\times Y))$.\footnote{Of course by $E \otimes F$ we mean the external tensor product $E \boxtimes F$ by abuse of notation.} \end{thm} 

\begin{proof} 
First we observe that if $X=S\cup T$ is the union of two closed subvarieties, and if $A$ generates $\Db\Coh(S)$ and $B$ generates $\Db\Coh(T)$, then $A\oplus B$ generates $\Db\Coh(X)$. We show that $\Db\Coh(S)$ and $\Db\Coh(T)$ together generate $\Db\Coh(X)$. Let $I_S$ be the sheaf of ideals of $S$ and let $I_T$ be the sheaf of ideals of $T$. Let $F$ be a coherent sheaf on $X$. Then we have the short exact sequence
\[ 0 \to I_S F \to F \to F/I_S F \to 0. \]
Since $I_S I_T = 0$, we see that $I_S F$ is a coherent sheaf on $T$ and $F/I_S F$ is a coherent sheaf on $S$. 
The claim follows.

Now to prove the theorem, we proceed by induction on $\dim X + \dim Y$. Let ${E}'$ be a generator of $\Db\Coh(\Sing(X))$ and ${F}'$ a generator of $\Db\Coh(\Sing(Y))$. 
By induction, we have that ${E}' \otimes {F}$ generates $\Db\Coh(\Sing(X)\times Y)$
and ${E}\otimes{F'}$ generates $\Db\Coh (X \times \Sing(Y))$. 
Let \[Z=(\Sing(X)\times Y) \cup (X\times \Sing(Y))\] which, because we are working over $\mathbb{C}$, is the same as $\Sing(X\times Y)$. 
Then $({E'}\otimes{F})\oplus({E}\otimes{F'})$ generates $\Db\Coh(Z)$. Let ${E}''$ be a generator of $\mathfrak{Perf}(X)$ and ${F}''$ a generator of $\mathfrak{Perf}(Y)$. Then $E''\otimes F''$ generates $\mathfrak{Perf}(X\times Y)$ by Lemma 3.4.1 and Theorem 2.1.2 of \cite{BonVDB}, and hence \[(E'' \otimes F'') \oplus (E'\otimes F) \oplus (E\otimes F')\] generates $\Db\Coh(X\times Y)$. Since each of the three summands is the external tensor product of sheaves, $E\otimes F$ generates $\Db\Coh (X\times Y)$ as desired.
 \end{proof}
 
\begin{rem} As a caution note that the hypothesis that the ground field be $\mathbb{C}$ is important here. The problem is illustrated by the fact that over an imperfect field $k$, it can happen that $X$ and $Y$ are regular but $X\times Y$ is not. Thus $\Db\Coh(X)$ and $\Db\Coh(Y)$ can have perfect generators whose external tensor product will fail to generate $\Db\Coh(X\times Y)$. As a consequence, the authors don't know of a clean statement for \prettyref{thm:functor} that
works over an arbitrary base field. 
\end{rem}

\begin{lem} 
\label{lem:duality}
We have a functor $D$ which takes a matrix factorization $P$ to the matrix factorization $\sheafHom_{\OO_X} (P,\OO_X)$ and which induces an equivalence between the categories $[\MFdg(X,W)]$ and $[\MFdg(X,-W)^\op]$. The functor
\[ \sheafRHom(-,\OO_{X_0}[1]) : \Db\Coh(X_0) \to \Db\Coh(X_0)^\op \]
induces a functor
\[ \DSing(X_0) \to \DSing(X_0)^\op \]
which we will also denote by $\sheafRHom(-,\OO_{X_0}[1])$. Then finally we have the following commutative diagram \[
\xymatrix { 
\Dabs \mf(X,W) \ar[d]^D \ar[r] & \DSing(X_0) \ar[d]^{\sheafRHom(-,\OO_{X_0}[1])} \\
\Dabs \mf(X, -W)^\op \ar[r] & \DSing(X_0)^\op.}\]
\end{lem}

%%%%%%%%%%%%%%%

For a dg category $T$, we recall the notation $\widehat{T} = \Int(T^{\op}\text{-}\Mod)$, the full dg subcategory of $T^{\op}\text{-}\Mod$ consisting of those $T^{\op}$-modules that are both fibrant and cofibrant (see \cite{Toen}).

\begin{proof}[Proof of \prettyref{thm:functor}] 
Let $E$ and $F$ be generators of $\Db\Coh(\Sing(W_1^{-1}(0)))$ and $\Db\Coh(\Sing(W_2^{-1}(0)))$ respectively. Let $P$ be a matrix factorization of $(X_1,W_1)$ such that we have a triangle \[P \to E \to C\] with $C$ acyclic, and similarly let $Q$ be a matrix factorization of $(X_2,-W_2)$ such that we have a triangle \[Q \to F \to C'\] with $C'$ acyclic --- we can do this by \prettyref{lem:coker}.
Let $A$ and $B^{\op}$ denote $\RHom(P,P)$ and $\RHom(Q,Q)$ respectively. Following the same argument as the proof of Theorem 4.2 of \cite{Dyckerhoff}, we know that 
$\Inj(X_1,W_1)\cong \widehat{A}$
and we know that 
$\Inj(X_2,-W_2)\cong \widehat{B^{\op}}$. We also have $\Inj(X_2, W_2) \cong \widehat{B}$.

The cone of $P \otimes Q \to E \otimes F$ is acyclic. By \prettyref{thm:generator}, $E \otimes F$ generates the category $\DSing(W^{-1}(0))$ where $W = \pi_1^\ast(W_1) - \pi_2^\ast(W_2)$, because we have 
\[\Sing(W^{-1}(0)) = \Sing(W_1^{-1}(0)) \times \Sing(W_2^{-1}(0)).\] (This equality is the main observation of this proof; all other parts of this proof are essentially standard considerations of Morita theory.) Therefore by \prettyref{lem:coker}, it follows that $P \otimes Q$ generates the matrix factorization category $\MFdg(X_1 \times X_2 , W)$.

Since $P$ and $Q$ are curved vector bundles, we have a canonical isomorphism 
\[\sheafHom(P,P) \otimes \sheafHom(Q,Q) \to \sheafHom(P \otimes Q , P \otimes Q).\]
We then have
\begin{align*}
\Hom_{\MFCech} ( P \otimes Q  , P \otimes Q ) & = \Gamma \Tot \mathcal{C}^\bullet( \mathfrak{U}, \sheafHom(P \otimes Q, P \otimes Q)) \\
&  \cong \Gamma \Tot \mathcal{C}^\bullet ( \mathfrak{U} , \sheafHom(P,P ) \otimes \sheafHom(Q,Q)) \\
&  \cong \Hom_{\MFCech} (P, P ) \otimes \Hom_{\MFCech} (Q,Q).
\end{align*}
Therefore we have \[\Inj (X_1 \times X_2, W) \cong \widehat{A \otimes B^{\op}}.\]
We conclude with the following string of isomorphisms \cite{Toen}:
\[
\Inj(X_1\times X_2,W) \cong \widehat{A \otimes B^{\op}}\cong \RHom_{\c}(\widehat{A},\widehat{B})\cong \RHom_{\c}(\Inj(X_1,W_1),\Inj(X_2,W_2)).
\]

In the case of $X_1 = X_2$, the claimed identification of the identity functor with $\Delta$ comes from the fact that $\RHom(P\otimes D(P),\Delta) \cong \RHom(P,P) = A$. The proof of this is the same as the proof of Proposition 6.3 of \cite{Dyckerhoff}.
\end{proof}

\begin{cor} By the above calculations and Corollary 1.24 of \cite{Orlov}, we conclude that when the critical locus of $W$ is proper, the category $\Inj(X,W)$ is dg affine, proper, and homologically smooth as a differential $\ZZ/2\ZZ$-graded category \cite{KKP}. \end{cor}

\begin{proof}
The identity functor corresponds to the diagonal curved sheaf, which is compact. Therefore the identity functor is a compact object in the endofunctor category $\RHom_{\c}(\Inj(X,W),\Inj(X,W))$, so $\Inj(X,W)$ is homologically smooth.
\end{proof}

\begin{lem} With the same assumptions as the previous corollary, we have the following result $$\RHom(\Inj(X_1,W_1)_\c,\Inj(X_2,W_2)_\c)\cong \Inj(X_1\times X_2,\pi_1^*(W_1)-\pi_2^*(W_2))_\c.$$  \end{lem}

\begin{proof} Both $\Inj(X_1,W_1)$ and $\Inj(X_2,W_2)$ are equivalent to $\widehat{A}$ and $\widehat{B}$, where $A$ and $B$ are smooth and proper dg algebras. What we need to know is that if  $M$ is an $A\otimes B^{\op}$-module such that for any perfect $A$-module $P$, in particular $A$ itself, $P\otimes M$ is perfect as a $B$-module, then $M$ is perfect. This follows immediately from the following well-known lemma, see e.g. Proposition 3.4 of \cite{Shklyarov}.
\end{proof}

\begin{lem} A module $N$ over a smooth and proper dg algebra over $k$ is perfect if and only if $ \dim_k H^\bullet (N)$ is finite. \end{lem}

\section{Calabi--Yau property}

The goal of this section is to prove the following theorem: 
\begin{thm} Let $(X,W)$ be as above and, in addition, suppose $X$ is Calabi--Yau. Then the category $\Inj(X,W)_\c$ is a Calabi--Yau category of dimension $n$, where $n$ is the dimension of $X$.\end{thm} 

\begin{proof}
As above, let $\tilde{W}$ be the function $\pi_1^*(W)-\pi_2^*(W)$ on $X\times X$. Denote $\tilde{W}^{-1}(0)$ by $S$. In the previous section we have proved that $\Inj(X,W) \cong \widehat{A}$, where $A=\RHom(P,P)$ and $P$ is a compact generator. Let $A^\e=A\otimes A^{\op}$ and recall that the inverse Serre bimodule is defined as $$A^! = \RHom_{(A^\e)^{\op}}(A,A^\e).$$ Thus to prove the Calabi--Yau property it suffices to prove that we have $A^! \cong A[n]$ in the category $[\Int(A^\e\text{-}\Mod)]$.

We need to recall some theory from \cite{RD}. First we recall that given a closed immersion $i: X \to Y$ there is a functor $$i^\flat  = \sheafRHom_Y(i_*\OO_X,-) : \Db\Coh(Y)) \to \Db\Coh(X)).$$
It is easy to check that this functor has the property that given two morphisms $i$ and $j$, we have $(j\circ i)^\flat\cong i^\flat \circ j^\flat$.
Now we can factor the diagonal morphism $\Delta: X\to X\times X$ as the composition of $i: X\to S $ and $j:S \to X\times X$, so by the Fundamental Lemma on page 179 of \cite{RD}, $$\Delta^\flat(\OO_{X\times X})= \sheafRHom_{X\times X}(\OO_\Delta,\OO_{X\times X})=(\OO_\Delta)\otimes\omega _{X/\mathbb{C}}^{\vee}[n],$$ where $\omega_{X / \mathbb{C}}$ is the canonical sheaf. The right hand side is $\OO_\Delta[n]$ when $X$ is Calabi--Yau. A simple calculation shows that $j^\flat(\OO_{X\times X})= \OO_S[1]$. Thus we conclude that $$\RHom_ S(\OO_\Delta,\OO_S[1])= \OO_\Delta[n].$$

From here this argument follows exactly the argument of Lemma 6.8 of \cite{Dyckerhoff}. We repeat it here to show how to adapt it to our situation. Consider $D(P)\otimes P$, which is a generator for the category $MF(X\times X, \tilde{W})$. 
For any $Z$, we have 
\[\RHom(D(P)\otimes P , Z) \cong \RHom(D(Z), P \otimes D(P)).\]

Now we let $Z$ be the diagonal shifted by (the parity of) the dimension of $X$.
By the discussion above and \prettyref{lem:duality}, $D(Z)$ corresponds to the diagonal $\Delta$.  Finally, we conclude with the following sequence of isomorphisms:
\begin{align*}
A[n] & \cong \RHom(D(P) \otimes P , Z) \quad \text{ ($\Delta$ corresponds to the identity)}  \\
 & \cong \RHom(D(Z), P \otimes D(P)) \quad \text{ ($D$ is a contravariant equivalence)}  \\
 & \cong \RHom_{(A^\e)^{\op}}(\RHom(P\otimes D(P),D(Z)),A^\e) \quad \text{ ($P \otimes D(P)$ is a generator)} \\
 & \cong \RHom_{(A^\e)^{\op}}(A,A^\e) \quad \text{ ($\Delta$ corresponds to the identity)} \\
 & = A^{!}.
\end{align*}
This completes the proof of the theorem.
\end{proof}

\section*{Acknowledgements}

We thank our advisor Constantin Teleman for his guidance and support. We would also like to thank Leonid Positselski for his interest in our work and for many useful comments, and Raphael Rouquier for helpful correspondence.

\end{document}